\documentclass[12pt,a4paper]{article} 

\usepackage{hyperref}
\hypersetup{colorlinks=true,linkcolor=red,citecolor=green,filecolor=magneta,urlcolor=cyan}

\usepackage{amsfonts,amssymb,amsmath,amsthm}
\usepackage[a4paper,hmargin={2.5cm,2.5cm},vmargin={2.3cm,2.3cm}]{geometry}

\usepackage{graphicx}

\usepackage[latin1]{inputenc}
\usepackage{color}

\newcommand{\condrel}{\kappa_r}
\newcommand{\condabs}{\kappa_a}

\newcommand{\Rset}{\mathbb{R}}
\newcommand{\Cset}{\mathbb{C}}
\newcommand{\Zset}{\mathbb{Z}}
\newcommand{\Nset}{\mathbb{N}}
\newcommand{\Qset}{\mathbb{Q}}

\newcommand{\iu}{\ensuremath{\mathrm{i}}}
\newcommand{\du}{\ensuremath{\mathrm{d}}}

\newcommand{\Arg}{\mathop{\rm Arg}}

\newtheorem{thm}{Theorem}
\newtheorem{prop}{Proposition}
\theoremstyle{definition}
\newtheorem{defn}[thm]{Definition}

\begin{document}

\title{Computing the matrix Mittag--Leffler function with applications to fractional calculus 
\footnote{This is the post-print of the paper: R.Garrappa, M.Popolizio, ``Computing the matrix Mittag--Leffler function with applications to fractional calculus'', \emph{Journal of Scientific Computing} (Springer), October 2018, Volume 77, Issue 1, pp 129-153, doi: 10.1007/s10915-018-0699-5, avialable at \url{https://doi.org/10.1007/s10915-018-0699-5}. 
This work is supported under the GNCS-INdAM 2017 project ``Analisi numerica per modelli descritti da operatori frazionari''.}
}

\author{
Roberto Garrappa \\
\small Universit\`a degli Studi di Bari, Dipartimento di Matematica, Bari, Italy\\ 
\small Member of the INdAM Research group GNCS \\
\small \texttt{roberto.garrappa@uniba.it}
\and
Marina Popolizio \\
\small Universit\`a del Salento, Dipartimento di Matematica e Fisica ``Ennio De Giorgi''. Lecce, Italy \\ \small Member of the INdAM Research group GNCS \\
\small \texttt{marina.popolizio@unisalento.it}
}

\date{Submitted on June 5$^{\text{th}}$, 2017} 

\maketitle

\begin{abstract}
The computation of the Mittag-Leffler (ML) function with matrix arguments, and some applications in fractional calculus, are discussed. In general the evaluation of a scalar function in matrix arguments may require the computation of derivatives of possible high order depending on the matrix spectrum. Regarding the ML function, the numerical computation of its derivatives of arbitrary order is a completely unexplored topic; in this paper we address this issue and three different methods are tailored and investigated. The methods are combined together with an original derivatives balancing technique in order to devise an algorithm capable of providing high accuracy. The conditioning of the evaluation of matrix ML functions is also studied. The numerical experiments presented in the paper show that the proposed algorithm provides high accuracy, very often close to the machine precision.

\end{abstract}




\section{Introduction}

When the Swedish mathematician Magnus Gustaf Mittag-Leffler introduced, at the beginning of the twentieth century, the function that successively would have inherited his name \cite{MittagLeffler1904,MittagLeffler1905}, he perhaps ignored the importance it would have gained several years later; indeed, although the introduction of the Mittag-Leffler (ML) function was motivated just by the analysis of divergent series, this function has nowadays a fundamental role in the theory of operators of fractional (i.e., non integer) order and it has become so important in this context as to be defined the ``Queen function of fractional calculus'' \cite{GorenfloKilbasMainardiRogosin2014,GorenfloMainardi1997,MainardiMuraPagnini2010}. 

For complex parameters $\alpha$ and $\beta$, with $\Re(\alpha) > 0$, the ML function is defined by means of the series
\begin{equation}\label{eq:ML}
	E_{\alpha,\beta}(z) = \sum_{j=0}^{\infty} \frac{z^{j}}{\Gamma(\alpha j + \beta)}
	, \quad z \in \Cset ,
\end{equation}
where $\Gamma(\cdot)$ is the Euler's gamma function. $E_{\alpha,\beta}(z)$ is an entire function of order $\rho=1/\Re(\alpha)$ and type $\sigma=1$ and it is clearly a generalization of the exponential function to which it reduces when $\alpha=\beta=1$ since for $j\in\Nset$ it is $\Gamma(j+1) = j!$. Throughout the paper we consider real values for $\alpha$ and $\beta$ which is the case of interest for common applications.

Despite the great interest in fractional calculus, few works have so far concerned the accurate evaluation of the ML function. Indeed this is in and of itself challenging and expensive and few efficient algorithms for this task have been devised only recently (e.g., see \cite{Garrappa2015_SIAM,GarrappaPopolizio2013,GorenfloLoutchkoLuchko2002,ValerioMachado2014,ZengChen2015}).

The ML function with matrix arguments is just as valuable as its scalar version and it can be successfully employed in several applications: for the efficient and stable solution of systems of fractional differential equations (FDEs), to determine the solution of certain multiterm FDEs, in control theory and in other related fields (e.g., see \cite{GarrappaMoretPopolizio2015,GarrappaMoretPopolizio2017,GarrappaPopolizio2011_MCS,MoretNovati2011,Popolizio2018,Rodrigo2016,WeidemanTrefethen2007}).

The aim of this paper is to discuss numerical techniques for the evaluation of the ML function with matrix arguments; in particular we are interested in methods working with high accuracy, if possible very close to the machine precision, in order to provide results which can be considered virtually exact in finite precision arithmetic.

Incidentally, and under a more general perspective, we cannot get out of mentioning the recent interest in the numerical approximation of matrix functions for applications in a wide range of areas, for which we refer the reader to the textbook by Higham \cite{Higham2008} and the readable papers \cite{DelBuonoLopezPoliti2008,FrommerSimoncini2008,HaleHighamTrefethen2008,HighamAlMohy2010}.

For general matrix functions the Schur--Parlett algorithm \cite{DaviesHigham2003} represents the most powerful method. It is based on the Schur decomposition of the matrix argument combined with the Parlett recurrence to evaluate the matrix function on the triangular factor \cite{GolubVanLoan1996}. Additionally, for the diagonal part, the Taylor series of the scalar function is used. This last task requires the knowledge of the derivatives of the underlying scalar function up to an order depending on the eigenvalues of the matrix argument. In particular, high orders are needed in the presence of multiple or highly clustered eigenvalues. This is a crucial issue which complicates the matrix case with respect to the scalar one.

Facing the evaluation of derivatives of the ML function, as requested by the Schur--Parlett algorithm, is a demanding task other than a rather unexplored topic (except, as far as we know, for one work dealing with just the first order derivative \cite{GorenfloLoutchkoLuchko2002}). A great portion of this paper is therefore devoted to discuss and deeply investigate three different methods for the evaluation of derivatives of the ML function; the rationale for introducing different methods is that each of them properly works in limited regions of the complex plane and for different parameter ranges. Thanks to our investigation, we are therefore able to tune a combined algorithm which applies, in an accurate way, to matrices with any eigenvalues location.

Besides, we analyze the conditioning of the computation of matrix ML functions to understand the sensitivity of the matrix function to perturbations in the data. We thus give a first contribution on this topic which we think can be of interest for readers interested in numerical applications of the matrix ML function, where rounding errors and perturbations are inevitable.

This paper is organized as follows: in Section \ref{S:Applications} we survey some of the most common applications of the ML function with matrix arguments. In Section \ref{S:MatrixFunctionSchurParlett} we discuss the generalization of a scalar function to matrix arguments and we review the Schur-Parlett algorithm for its numerical computation. Some methods for the accurate and efficient evaluation of derivatives of the ML function, as requested by the Schur-Parlett algorithm, are hence described in Section \ref{S:DerivativeML}: we study in detail each method in order to identify strengths and weaknesses and provide some criteria for the selection of the most suitable method in each situation. The conditioning of the ML function is studied in Section \ref{S:Conditioning} and, finally, in Section \ref{S:Experiments} we present the results of some numerical experiments.

\section{Applications of matrix Mittag-Leffler functions}\label{S:Applications}

The evaluation of matrix ML functions is not just a curiosity-driven problem; several practical applications can indeed benefit from calculating the value of $E_{\alpha,\beta}$ in matrix arguments. 

This section presents some important applications involving the numerical evaluation of matrix ML functions. The aim is to give just a flavor of the assorted fields in which these objects are required, while for their thorough discussion we refer the reader to the existing literature (e.g., see \cite{GorenfloKilbasMainardiRogosin2014,HauboldMathaiSaxena2011,Rogosin2015}).

\subsection{Solution of systems of FDEs} 

The matrix ML function is crucial to explicitly represent the solution of a linear system of FDEs of order $\alpha >0$
\begin{equation}\label{eq:FDE_Linear}
	D^{\alpha}_{t} Y(t) = A Y(t)  , \quad Y^{(\ell)}(0) = Y^{\ell}_0, \quad \ell=0,\ldots,m-1 ,
\end{equation}
where $A\in\Rset^{n\times n}$, $Y(t):[0,T] \to \Rset^n$, $m=\lceil \alpha\rceil$ is the smallest integer greater or equal to $\alpha$ and $D_t^{\alpha}$ is the Caputo's fractional derivative
\[
	D^{\alpha}_{t} y(t) \equiv
	\frac{1}{\Gamma(m-\alpha)}
	  \int_{0}^t \frac{y^{(m)}(u)}{\bigl(t-u\bigr)^{\alpha+1-m}} \du u ,
\]
with $y^{(m)}$ denoting the standard integer--order derivative.

It is immediate to verify that the exact solution of the system (\ref{eq:FDE_Linear}) is 
\begin{equation}\label{eq:FDE_Linear_Solution}
	Y(t) = \sum_{\ell=0}^{m-1} t^{\ell} E_{\alpha,\ell+1}\bigl(t^{\alpha} A\bigr) Y^{\ell}_0 
\end{equation}
and, once some tool for the computation of the matrix ML function is available, the solution of (\ref{eq:FDE_Linear}) can be evaluated, possibly with high accuracy, directly at any time $t>0$. 
Conversely, the commonly used step-by-step methods usually involve considerable computational costs because of the persistent memory of fractional operators \cite{Diethelm2010}.

\subsection{Solution of time-fractional partial differential equations} 

Given a linear time-fractional partial differential equation
\begin{equation}\label{eq:tFPDE_Linear}
	D^{\alpha}_{t} u(t,x) = \nabla^2_x u(t,x) + f(t,x) , 
\end{equation}
subject to some initial and boundary conditions, a preliminary discretization along the spatial variables allows to recast (\ref{eq:tFPDE_Linear}) in terms of a semi-linear system of FDEs in the form 
\begin{equation}\label{eq:FDE_SemiLinear}
	D^{\alpha}_{t} U(t) = A U(t) + F(t) ,
\end{equation}
with $U(t) = \bigl(u_1(t),\dots,u_{N_x}(t)\bigr)^T$, being $u_i(t)$ an approximation of $u(x_i,t)$ on a partition $x_1,\dots,x_{N_{x}}$ of the spatial domain, and $F(t)$ obtained from the source term and the boundary conditions. By assuming for simplicity $0<\alpha<1$, the exact solution of the semi-discretized system (\ref{eq:FDE_SemiLinear}) can be formulated as
\begin{equation}\label{eq:FDE_SemiLinear_Solution}
	U(t) = E_{\alpha,1}\bigl(t^{\alpha} A\bigr) U_0 + \int_{{0}}^{t} (t-\tau)^{\alpha-1}E_{\alpha,\alpha}((t-\tau)^{\alpha}A) F(\tau) \, \du \tau.
\end{equation}

A common approach to solve (\ref{eq:FDE_SemiLinear}) relies on product-integration rules which actually approximate, in the standard integral formulation of (\ref{eq:FDE_SemiLinear}), the vector-field $AU(t)+F(t)$ by piecewise interpolating polynomials. This method, however, has severe limitations for convergence due to the non-smooth behavior of $U(t)$ at the origin \cite{Dixon1985}. However, the same kind of approximation when used in (\ref{eq:FDE_SemiLinear_Solution}), where just $F(t)$ is replaced by polynomials, does not suffer from the same limitations \cite{Garrappa2013_EPJST} and it is possible to obtain high order methods under the reasonable assumption of a sufficiently smooth source term $F(t)$. In addition, solving (\ref{eq:tFPDE_Linear}) by using matrix ML functions in (\ref{eq:FDE_SemiLinear_Solution}) allows also to overcome stability issues since the usual stiffness of the linear part of (\ref{eq:FDE_SemiLinear}) is solved in a virtually exact way.

\subsection{Solution of linear multiterm FDEs}


Let us consider a linear multiterm FDE of \emph{commensurate} order $0<\alpha<1$ 
\begin{equation}\label{eq:MultiTermFDE}
	\sum_{k=0}^{n} a_{k} \, D^{k \alpha}_{t} y(t) = f(t) ,
\end{equation}
with $D^{k \alpha}_{t}$ derivatives of Caputo type and associated initial conditions 
\[
	y(0) = b_0 , \, y'(0) = b_1, \, \dots , y^{(m-1)}(0)=b_{m-1} , \quad m = \left\lceil n \alpha \right\rceil .
\]

When $n \alpha<1$, or when $\alpha\in\Qset$, the FDE (\ref{eq:MultiTermFDE}) can be reformulated as a linear system of FDEs (see \cite{DiethelmFord2002_BIT,DiethelmFord2004_AMC} or \cite[Theorems 8.1-8.2]{Diethelm2010}). For instance, if $\alpha=p/q$, with $p,q\in\Nset$, 
 one puts $y_{1}(t) = y(t), \,y_{k}(t) = D^{1/q}_t y_{k-1}(t)$ and, by introducing the vector notation $Y(t) = \bigr(y_{1}(t),y_{2}(t),\dots,y_{N}(t)\bigr)^T$, it is possible to rewrite (\ref{eq:MultiTermFDE}) as
\begin{equation}\label{eq:MultiTermFDESystem}
	D^{1/q}_{t} Y(t) = A Y(t) + e_{N} f(t) , \quad Y(0) = Y_0
\end{equation}
where $e_{N} = \bigl(0,0,\dots,0,1\bigr)^T\in\Rset^{N}$, $Y_0\in\Rset^{N}$ is composed in a suitable way on the basis of the initial values $b_j$, the coefficient matrix $A\in \Rset^{N\times N}$ is a companion matrix and the size of problem (\ref{eq:MultiTermFDESystem}) is $N = np$.

Also in this case it is possible to express the exact solution by means of matrix ML functions as
\begin{equation}\label{eq:MultitermFDESolution}
	Y(t) = E_{\alpha,1}(t^{\alpha}A) Y_0 + \int_0^t (t-\tau)^{\alpha-1} E_{\alpha,\alpha}((t-\tau)^{\alpha}A) e_N f(\tau) \du \tau
\end{equation}
and hence approximate the above integral by some well-established technique. As a special case, with a polynomial source term $f(t)=c_0 + c_1 t + \dots c_s t^{s}$ it is possible to explicitly represent the exact solution of (\ref{eq:MultiTermFDESystem}) as
\begin{equation}\label{eq:MultitermFDESpectralSolution}
	Y(t) = E_{\alpha,1}(t^{\alpha}A) Y_0 + \sum_{\ell=0}^{s} \ell ! c_{\ell} t^{\alpha+\ell} E_{\alpha,\alpha+\ell+1}(t^{\alpha}A) e_N .
\end{equation}

We have also to mention that, as deeply investigated in \cite{LuchkoGorenflo1999}, the solution of the linear multiterm FDE (\ref{eq:MultiTermFDE}) can be expressed as the convolution of the given input function $f(t)$ with a special multivariate ML function of scalar type, an alternative formulation which can be also exploited for numerically solving (\ref{eq:MultiTermFDE}). 

It is beyond the scope of this work to compare the approach proposed in \cite{LuchkoGorenflo1999} with numerical approximations based on the matrix formulation (\ref{eq:MultitermFDESolution}); it is however worthwhile to highlight the possibility of recasting the same problem in different, but analytically equivalent, formulations: one involving the evaluation of scalar multivariate ML functions and the other based on the computation of standard ML functions with matrix arguments; clearly the two approaches present different computational features whose advantages and disadvantages could be better investigated in the future.

\subsection{Controllability and observability of fractional linear systems}

For a continuous-time control system the complete controllability describes the possibility of finding some input signal $u(t)$ such that the system can be driven, in a finite time, to any final state $x_{T}$ starting from the fixed initial state $x_0$. The dual concept of observability instead describes the possibility of determining, at any time $t$, the state $x_0=x(t_0)$ of the system from the subsequent history of the input signal $u(t)$ and of the output $y(t)$.

Given a linear time-invariant system of fractional order $\alpha>0$
\begin{equation}\label{eq:LinearFracSystem}
	\left\{\begin{array}{l}
		D^{\alpha}_{t} x(t) = A x(t) + B u(t) \\
		y(t) = C x(t) + D u(t) \\
	\end{array} \right. 
\end{equation}
controllability and observability are related to the corresponding Gramian matrices, defined respectively as \cite{BalachandranGovindarajOrtigueiraRiveroTrujillo2013,MatignonAndreaNovel1996} 
\[
	{\cal C}_{\alpha}(t)  := \int_0^t E_{\alpha,\alpha}((t-\tau)^{\alpha}A) B B^T E_{\alpha,\alpha}((t-\tau)^{\alpha}A^T) \du \tau 
\]
and
\[
	{\cal O}_{\alpha}(t) = \int_0^t E_{\alpha,1}(\tau^{\alpha}A^T) C^T C E_{\alpha,1}((t-\tau)^{\alpha}A) \du \tau 
\]
(see also \cite{MatychynOnyshchenko2015} for problems related to the control of fractional order systems).

In particular, system (\ref{eq:LinearFracSystem}) is controllable (resp. observable) on $[0 , T]$ if the matrix ${\cal C}_{\alpha}(t)$ (resp. the matrix ${\cal O}_{\alpha}(t)$) is positive definite.

Although alternative characterizations of controllability and observability are available in terms of properties of the state matrix $A$, the computation of the controllability Gramian ${\cal C}_{\alpha}(t)$ (resp. the observability Gramian ${\cal O}_{\alpha}(t)$) can be used to find an appropriate control to reach a state $x_{T}$ starting from the initial state $x_0$ (resp. to reconstruct the state $x(t)$ from the input $u(t)$ and the output $y(t)$). Moreover, computing the Gramians is also useful for input-output transfer function models derived from linearization of nonlinear state-space models around equilibrium points that depend on working conditions of the real modeled systems (see, for instance, \cite{LinoMaione2013b,LinoMaione2013a}).

\section{The Mittag-Leffler functions with matrix arguments: definitions}\label{S:MatrixFunctionSchurParlett}

\subsection{Theoretical background}

Given a function $f$ of scalar arguments and a matrix $A$, the problem of finding a suitable definition for $f(A)$ goes back to Cayley (1858) and it has been broadly analyzed since then. The extension from scalar to matrix arguments is straightforward for simple functions, like polynomials, since it trivially consists in substituting the scalar argument with the matrix $A$. This is also the case of some transcendental functions, like the ML, for which (\ref{eq:ML}) becomes
\begin{equation}\label{eq:ML_series_mat}
	E_{\alpha,\beta}(A) = \sum_{j=0}^{\infty} \frac{A^{j}}{\Gamma(\alpha j + \beta)}
\end{equation}
(this series representation is useful for defining $E_{\alpha,\beta}(A)$ but usually it cannot be used for computation since it presents the same issues, which will be discussed later in Subsection \ref{sec:series}, of the use of (\ref{eq:ML}) for the scalar function, amplified by the computation increase due to the matrix argument). 


In more general situations a preliminary definition is of fundamental importance to understand the meaning of evaluating a function on a matrix.

\begin{defn}
Let $A$ be a $n\times n$ matrix with $s$ distinct eigenvalues $\lambda_1,\ldots,\lambda_s$ and let $n_i$ be the {\emph{index}} of $\lambda_i$, that is, the smallest integer $k$ such that $(A-\lambda_i I)^k=0$ with $I$ denoting the $n\times n$ identity matrix. Then the function $f$ is said to be defined on the spectrum of $A$ if the values $f^{(j)}(\lambda_i), \,\, j=0,\ldots,{n_i}-1,\,\,i=1,\ldots,s$ exist.
\end{defn}

Then for functions defined on the spectrum of $A$ the {\emph{Jordan canonical form}} can be of use to define matrix functions. 
\begin{defn}
Let $f$ be defined on the spectrum of $A$ and let $A$ have the Jordan canonical form
\[
A= Z J Z^{-1}= Z \operatorname{diag}(J_1,\ldots,J_p) Z^{-1} ,
\quad 
J_k=\left[
\begin{array}{cccc}
\lambda_k  & 1 & &  \\
 & \lambda_k& \ddots & \\
&      &        \ddots    & 1 \\
&      &    & \lambda_k
\end{array}
\right]
\in\Cset^{m_k \times m_k}.
\]

Then
\[
f(A)=Zf(J)Z^{-1}=Z\operatorname{diag}(f(J_1),\dots,f(J_p))Z^{-1} ,
\]
with
\[
f(J_k)=\left[
\begin{array}{cccc}
f(\lambda_k)  & f^{\prime}(\lambda_k) & \ldots & \frac{f^{(m_k-1)}(\lambda_k)}{(m_k-1)!} \\
 & f(\lambda_k)& \ddots & \vdots \\
&      &        \ddots    & f^{\prime}(\lambda_k) \\
&      &    & f(\lambda_k)
\end{array}
\right].
\]
\end{defn}

This definition highlights a fundamental issue related to matrix functions: when multiple eigenvalues are present, the function derivatives need to be computed. In our context this is a serious aspect to consider since no numerical method has been considered till now for the derivatives of the (scalar) ML function, except for the first order case \cite{GorenfloLoutchkoLuchko2002}. 

In practice, however, the Jordan canonical form is rarely used since the similarity transformation may be very ill conditioned and, even more seriously, it cannot be reliably computed in floating point arithmetic.

\subsection{Schur-Parlett algorithm}

Diagonalizable matrices are favorable arguments for simple calculations; indeed a matrix is diagonalizable if there exist a nonsingular matrix $P$ and a diagonal matrix $D=\text{diag}(\lambda_1,\ldots,\lambda_n)$ such that $A=P^{-1}DP$. In this case $f(A)=P^{-1}f(D)P$ and $f(D)$ is still a diagonal matrix with principal entries $f(\lambda_{i})$. Clearly the computation involves only scalar arguments. Problems arise when the matrix $P$ is nearly singular or in general badly conditioned, since the errors in the computation are related to the condition number of the similarity matrix $P$ (see \cite{Higham2008}). For these reasons well conditioned transformations are preferred in this context.

The Schur decomposition is a standard instrument in linear algebra and we refer to \cite{GolubVanLoan1996} for a complete description and for a comprehensive treatment of the related issues. It nowadays  represents the starting point for general purpose algorithms to compute matrix functions thanks also to its backward stability. Once the Schur decomposition for the matrix argument $A$ is computed, say $A=QTQ^*$ with $Q$ unitary and $T$ upper triangular, then $f(A)=Q f(T)Q^*$, for any well defined $f$. The focus thus moves to the computation of the matrix function for triangular arguments. 

In general this is a delicate issue which can be the cause of severe errors if not cleverly accomplished. Here we describe the enhanced algorithm due to Davies and Higham \cite{DaviesHigham2003}: once the triangular matrix $T$ is computed, it is reordered and blocked 
so to get a matrix $\tilde{T}$  such that each block ${\tilde{T}}_{ij}$ has clustered eigenvalues and distinct diagonal blocks have far enough eigenvalues. 

To evaluate  $f({\tilde{T}}_{ii})$ a Taylor series is considered about the mean $\sigma$ of the eigenvalues of ${\tilde{T}}_{ii}$. So, if $m$ is the dimension of ${\tilde{T}}_{ii}$ then $\sigma=\textrm{trace}({\tilde{T}}_{ii})/m$ and 
\[
f({\tilde{T}}_{ii})=\displaystyle\sum_{k=0}^{\infty} \frac{f^{(k)}(\sigma)}{k!} M^k
\]
with $M$ such that ${\tilde{T}}_{ii}=\sigma I +M$. The powers of $M$ will decay quickly after the $(m-1)$-st since, by construction, the eigenvalues of ${\tilde{T}}_{ii}$ are ``sufficiently'' close.

Once the diagonal blocks  $f({\tilde{T}}_{ii})$ are evaluated, the rest of $f(\tilde{T})$ is computed by means of the block form of the Parlett recurrence \cite{Higham2008}. Finally the inverse similarity transformations and the reordering are applied to get $f(A)$.

\section{Computation of derivatives of the ML function}\label{S:DerivativeML}

The Schur-Parlett algorithm substantially relies on the knowledge of the derivatives of the scalar function in order to compute the corresponding matrix function. Since in the case of the ML function $E_{\alpha,\beta}(z)$ the derivatives are not immediately at disposal, it is necessary to devise efficient and reliable numerical methods for their computation for arguments located in any region of the complex plane and up to any possible order, in dependence of the spectrum of $A$.

To this purpose we present and discuss three different approaches, namely 
\begin{enumerate}
	\item series expansion,
	\item numerical inversion of the Laplace transform,
	\item summation formulas.
\end{enumerate}

Although the three approaches are based on formulas which are equivalent from an analytical point of view, their behavior when employed for numerical computation differs in a substantial way. 

We  therefore investigate in details weaknesses and strengths in order to identify the range of arguments for which each of them provides more accurate results and we tune, at the end of this Section, an efficient algorithm implementing, in a combined way, the different methods.

\subsection{Series expansion}\label{sec:series}

After a term-by-term derivation, it is immediate to construct from (\ref{eq:ML}) the series representation of the derivatives of the ML function 
\begin{equation}\label{eq:ML_Deriv}
	\frac{d^k}{dz^k} E_{\alpha,\beta}(z) 
	= \sum_{j=k}^{\infty} \frac{(j)_k }{\Gamma(\alpha j + \beta)} z^{j-k}
	, \quad
	k \in \Nset ,
\end{equation}
with $(x)_k$ denoting the falling factorial
\[
(x)_k=x(x-1)\cdots (x-k+1) .
\]

At a glance, the truncation of (\ref{eq:ML_Deriv}) to a finite number $J \in \Nset$ of terms, namely
\begin{equation}\label{eq:ML_DerivTruncated}
	\frac{d^k}{dz^k} E_{\alpha,\beta}(z) 
	\approx \sum_{j=k}^{J} \frac{(j)_k }{\Gamma(\alpha j + \beta)} z^{j-k} ,
\end{equation}
may appear as the most straightforward method to compute derivatives of the ML function. Anyway the computation of (\ref{eq:ML_DerivTruncated}) presents some not negligible issues entailing serious consequences if not properly addressed.


We have first to mention that in IEEE-754 double precision arithmetic the Gamma function can be evaluated only for arguments not greater than 171.624, otherwise overflow occurs; depending on $\alpha$ and $\beta$ the upper bound for the number of terms in (\ref{eq:ML_DerivTruncated}) is hence 
\[
 J_{\max} := \left\lfloor \frac{171.624-\beta}{\alpha} \right\rfloor ,
\]
with $\left\lfloor x \right\rfloor$ being the largest integer less than or equal to $x$; as a consequence the truncated series (\ref{eq:ML_DerivTruncated}) can be used only when the $J_{\max}$-th term is small enough to disregard the remaining terms in the original series (\ref{eq:ML_Deriv}). By assuming a target accuracy $\tau$ as the threshold for truncating the summation, a first bound on the range of admissible values of $z$ is then obtained by requiring that
\[
	|z| \le \left( \tau \frac{ \Gamma(\alpha J_{\max} + \beta)}{(J_{\max})_k}\right)^{\frac{1}{J_{\max}-k}} .
\]

The effects of round-off errors (related to the finite-precision arithmetic used for the computation) may however further reduce, in a noticeable way, the range of admissible arguments $z$ for which the series expansion can be employed. When $|z|>1$ and $|\arg(z)| > \alpha \pi /2$, some of the terms in (\ref{eq:ML_DerivTruncated}) are indeed in modulus much more larger than the modulus of $E_{\alpha,\beta}(z)$; actually, small values of $E_{\alpha,\beta}(z)$ result as the summation of large terms with alternating signs, a well-known source of heavy numerical cancellation (e.g., see \cite{Higham2002}).

To decide whether or not to accept the results of the computation, it is necessary to devise a reliable estimate of the round-off error. We inherently assume that in finite-precision arithmetic the sum $S_1=c_0+c_1$ of any two terms $c_0$ and $c_1$ leads to $\hat{S}_1=(c_0+c_1)(1+\delta_1)$, with $|\delta_1| < \epsilon$ and $\epsilon>0$ the machine precision. Then, for $J\geq 2$, the general summation $S_{J}=c_0+c_{1}+\ldots +c_{J}$ of $J+1$ terms is actually computed as 
\begin{equation}\label{eq:SumFinitePrecision}
	\hat{S}_J = S_J + (c_0 + c_1) \sum_{j=1}^{J} \delta_j + c_2 \sum_{j=2}^{J} \delta_j + c_{3} \sum_{j=3}^{J} \delta_j + \dots + c_{J} \delta_J 
\end{equation}
where $|\delta_1|,|\delta_2|,\dots,|\delta_J|<\epsilon$ and terms proportional to ${\cal{O}}({\epsilon}^2)$ have been discarded. It is thus immediate to derive the following bound for the round-off error
\begin{equation}\label{eq:BoundSumFinitePrecision1}
	\bigl| S_J - \hat{S}_J \bigr| \le \epsilon \left( J |c_{0}| + \sum_{j=1}^{J} (J-j+1) \bigl| c_j \bigr| \right) .
\end{equation}

The order by which the terms $c_j$ are summed is relevant especially for the reliability of the above estimator; clearly, an ascending sorting of $|c_j|$ makes the estimate (\ref{eq:BoundSumFinitePrecision1}) more conservative and hence more useful for practical use. It is therefore advisable, especially in the more compelling cases (namely, for arguments $z$ with large modulus and $|\arg(z)| > \alpha \pi /2$) to perform a preliminary sorting of the terms in (\ref{eq:ML_DerivTruncated}) with respect to their modulus.

An alternative estimate of the round-off error can be obtained after reformulating (\ref{eq:SumFinitePrecision}) as
\[
	\hat{S}_J = S_J + \sum_{j=1}^{J} S_j \delta_j 
\]
from which one can easily derive
\begin{equation}\label{eq:BoundSumFinitePrecision3}
	\bigl| S_J - \hat{S}_J \bigr| \le \epsilon \sum_{j=1}^{J} \bigl| S_j \bigr| .
\end{equation}

The bound (\ref{eq:BoundSumFinitePrecision3}) is surely more sharp than (\ref{eq:BoundSumFinitePrecision1}) and it does not involve a large amount of extra computation since it is possible to evaluate $S_J$ by storing all the partial sums $S_j$; anyway (\ref{eq:BoundSumFinitePrecision3}) is based on the exact values of the partial sums $S_j$ which are actually not available; since just the computed values $\hat{S}_j$ can be employed, formula (\ref{eq:BoundSumFinitePrecision3}) can underestimate the round-off error. We find useful to use a mean value between the bounds provided by (\ref{eq:BoundSumFinitePrecision1}) and (\ref{eq:BoundSumFinitePrecision3}).

Unfortunately, only ``a posteriori'' estimates of the round-off error are feasible and it seems not possible to prevent from performing the computation before establishing whether or not to accept the results. Obviously, on the basis of the above discussion, it is possible to reduce the range of arguments for which the evaluation by (\ref{eq:ML_DerivTruncated}) is attempted and, anyway, if well organized, the corresponding algorithm runs in reasonable time.

\subsection{Numerical inversion of the Laplace transform}\label{SS:NumericalLT}

After a change of the summation index, it is straightforward from (\ref{eq:ML_Deriv}) to derive the following representation of the derivatives of the ML function 
\begin{equation}\label{eq:MLder_ML3}
	\frac{d^k}{dz^k} E_{\alpha,\beta}(z) 
	= \sum_{j=0}^{\infty} \frac{(j+k)_k z^{j}}{\Gamma(\alpha j + \alpha k + \beta)} 
	= k! E_{\alpha,\alpha k + \beta}^{k+1}(z) ,
\end{equation}
where 
\begin{equation}\label{eq:ML3}
	E_{\alpha,\beta}^{\gamma}(z) = \frac{1}{\Gamma(\gamma)} \sum_{j=0}^{\infty} \frac{\Gamma(j+\gamma) z^{j}}{j!\Gamma(\alpha j + \beta)} 
\end{equation}
is a three parameter ML function which is recognized in literature as the Prabhakar function \cite{Prabhakar1971} (the relationship between the derivatives of the ML function and the Prabhakar function has been recently highlighted also in \cite{GarraGarrappa2018_CNSNS,OliveiraOliveiraDeif2016}). 

This function, which is clearly a generalization of the ML function (\ref{eq:ML}), since $E_{\alpha,\beta}^{1}(z)=E_{\alpha,\beta}(z)$, is recently attracting an increasing attention in view of its applications in the description of relaxation properties in a wide range of anomalous phenomena (see, for instance \cite{ColombaroGiustiVitali2018,GarrappaMainardiMaione2016,GiustiColombaro2018,LiemertSandevKantz2017,Sandev2017}). For this reason some authors recently investigated the problem of its numerical computation \cite{Garrappa2015_SIAM,StanislavskyWeron2012}.

The approach proposed in \cite{Garrappa2015_SIAM} is based on the numerical inversion of the Laplace transform and extends a method previously proposed in \cite{GarrappaPopolizio2013,WeidemanTrefethen2007}. Since it allows to perform the computation with an accuracy very close to the machine precision, the extension to derivatives of the ML function appears of particular interest.

For any $t>0$ and $z \in \Cset$, the Laplace transform (LT) of the function $t^{\beta-1} E_{\alpha,\beta}^{\gamma}(t^{\alpha} z)$ is given by 
\begin{equation}\label{eq:ML3_LT}
	{\cal L} \left( t^{\beta-1} E_{\alpha,\beta}^{\gamma}(t^{\alpha} z) \, ; \, s \right)
	= \frac{s^{\alpha \gamma-\beta}}{(s^{\alpha} - z)^{\gamma}}
	, \quad \Re(s)>0, \,\, |z s^{-\alpha}| < 1.
\end{equation}
After selecting $t=1$ and denoting
\[
	H_{k}(s;z) = \frac{s^{\alpha - \beta}}{(s^{\alpha} - z)^{k+1}} ,
\]
with a branch--cut imposed on the negative real semi axis to make $H_{k}(s;z)$ single-valued, the derivatives of the ML function can be evaluated by inverting the LT (\ref{eq:ML3_LT}), namely by recasting (\ref{eq:MLder_ML3}) in form of the integral
\begin{equation}\label{eq:ML3_der_Integral}
	\frac{d^k}{dz^k} E_{\alpha,\beta}(z)  = \frac{k!}{2\pi i} \int_{{\cal C}} e^{s} H_{k}(s;z) \, \du s =\mathrel{\mathop:} I_k(z) ,
\end{equation}
over a contour ${\cal C}$ in the complex plane encompassing at the left all the singularities of $H_k(s;z)$.

The  trapezoidal rule has been proved to possess excellent properties for the numerical quadrature of contour integrals of analytic functions \cite{BornemannLauriWagoWaldvogel2004,TrefethenWeideman2014} but satisfactory enough results can be obtained also in the presence of singularities if the contour is selected in a suitable way.

In particular, in order to truncate, after a reasonable small number of terms, the infinite sum resulting from the application of the trapezoidal rule to (\ref{eq:ML3_der_Integral})  it is necessary that the integrand decays in a fast way; therefore contours beginning and ending in the left half of the complex plane must be selected. Our preference is for parabolic contours 
\[
	{\cal C} : \sigma(u) = \mu (i u +1)^2 , \quad -\infty < u < \infty
\]
whose simplicity allows accurate estimates of the errors and a fine tuning of the main parameters involved in the integration.

Given a grid $u_{j}=jh$, $j=-N,\dots,N$, with constant step-size $h>0$, the truncated trapezoidal rule applied to (\ref{eq:ML3_der_Integral}) reads as
\[
	I_k^{[N]}(z) = \frac{k! h}{2\pi i} \sum_{j=-N}^{N} e^{\sigma(u_{j})} H_{k}(\sigma(u_{j});z) \sigma'(u_{j}) .
\]

For choosing the integration parameters $\mu$, $h$ and $N$ the two main components of the error are considered, that is the discretization and the truncation errors. The balancing of these errors, according to the procedure described first in \cite{WeidemanTrefethen2007}, and hence applied to the ML function in \cite{Garrappa2015_SIAM,GarrappaPopolizio2013}, allows to select optimal parameters in order to achieve a prescribed tolerance $\tau > 0$.

The error analysis is performed on the basis of the distance between the contour and the integrand singularities. As observed in \cite{Garrappa2015_SIAM}, in addition to the branch-point singularity at $s=0$, the set of the non zero poles $s^{\star}$ of $H_{k}(s;z)$ in the main Riemann sheet, i.e. such that $-\pi < \Arg(s^{\star}) \le \pi$, is given by
\begin{equation}\label{eq:SetPoles}
	S^{\star} = \left\{ | s |^{1/\alpha} e^{i\frac{\Arg(s) + 2j\pi}{\alpha}} \, \Bigr| \,
	  -\frac{\alpha}{2} - \frac{\Arg(s)}{2\pi} < j \le \frac{\alpha}{2} - \frac{\Arg(s)}{2\pi} , \, j \in \Zset \right\} .
\end{equation}

Depending on $\alpha$ and $z$, one can therefore expect singularities in any number and in any region of the complex plane but the selection of contours so wide to encompass all the singularities usually leads to numerical instability. In order to avoid that round-off errors dominate discretization and truncation errors, thus preventing from achieving the target accuracy, it is indeed necessary to select contours which do not extend too far in the right half-part of the complex plane; in particular, it has been found \cite{Garrappa2015_SIAM} that the main parameter $\mu$ of the parabolic contour must satisfy 
\begin{equation}\label{eq:VincoloMu}
	\mu < \log\tau - \log\epsilon,
\end{equation}
with $\epsilon$ the machine precision and $\tau> \epsilon$ the target accuracy.

Since in many cases it is not possible to leave all the singularities at the left of the contour and simultaneously satisfy (\ref{eq:VincoloMu}), some of the singularities must be removed from (\ref{eq:ML3_der_Integral}) to gain more freedom in the selection of the contour ${\cal C}$; the corresponding residues are therefore subtracted in view of the Cauchy's residue theorem and hence 
\begin{equation} \label{eq:GeneralizedMLInverseLaplace}
	\frac{d^k}{dz^k} E_{\alpha,\beta}(z) = 
		k! \sum_{s^{\star} \in S^{\star}_{\cal C}} \operatorname{Res} \bigl( e^{s} H_k(s;z) , s^{\star} \bigr) +
		\frac{k!}{2\pi i} \int_{{\cal C}} e^{s} H_k(s;z)  \du s,
\end{equation}
where $S^{\star}_{\cal C} \subseteq S^{\star} $ is the set of the singularities of $H_k(s;z)$ laying to the right of ${\cal C}$ (which now is not constrained to encompass all the singularities) and $\operatorname{Res} \bigl( e^{s} H_k(s;z) , s^{\star} \bigr)$ is the residue of $e^{s} H_k(s;z) $ at $s^{\star}$; since the selected branch--cut, the singularity at the origin is always kept at the left of ${\cal C}$.

To make equation (\ref{eq:GeneralizedMLInverseLaplace}) applicable for computation we give an explicit representation of the residues in terms of elementary functions by means of the following result.

\begin{prop}\label{prop:ResiduesH}
Let $\alpha>0$, $\beta \in \Rset$, $k \in \Nset$, $z \not=0$ and $s_{\star}\in S_{\star}$ one of the poles of $H_k(s;z)$. Then
\[
	\operatorname{Res} \bigl( e^{s}H_{k}(s;z),s_{\star} \bigr) 
	= \frac{1}{\alpha^{k+1}} e^{s_{\star}} s_{\star}^{1- \alpha k - \beta} P_{k}( s_{\star})
\]
where $P_{k}(x)=p_0^{(k)} + p_1^{(k)} x + \dots p_{k}^{(k)} x^{k}$ is the $k$-th degree polynomial whose coefficients are 
\[
	p_{j}^{(k)} = 
		\frac{1}{j!} 
		\sum_{\ell=0}^{k-j} 
		\frac{ (\alpha - \beta)_{\ell}  }{\ell! }  H^{(k)}_{k-j-\ell}
	, \quad j = 0, 1, \dots , k 
\]
with the coefficients $H^{(k)}_{j}$ evaluated recursively as
\[
	H^{(k)}_{0} = 1 
	, \quad
	H^{(k)}_{j} = - \frac{1}{\alpha} \sum_{\ell=1}^{j} \binom{\alpha}{\ell+1} \left(\frac{k\ell}{j} + 1 \right) H^{(k)}_{j-\ell}
	, \quad j=1,2,\dots,k 
\]
and the generalized binomial coefficients are defined as
\[
	\binom{\alpha}{j} = \frac{\alpha(\alpha-1)(\alpha-2)\cdots(\alpha-j+1)}{j!} 
	= (-1)^{j} \frac{\Gamma(j-\alpha)}{\Gamma(-\alpha)\Gamma(j+1)}.
\]
\end{prop}

\begin{proof}
Since $s_{\star}$ is a pole of order $k+1$ of $e^{s}H_{k}(s;z)$, the corresponding residue can be evaluated by differentiation
\begin{equation}\label{eq:ResidueDifferentiation}
	\operatorname{Res} \bigl( e^{s}H_{k}(s;z),s_{\star} \bigr) = \frac{1}{k!} \lim_{s \to s_{\star}} \frac{d^{k}}{ds^{k}}  e^{s}H_{k}(s;z) (s-s_{\star})^{k+1} .
\end{equation}

To study the term $H_{k}(s;z) (s-s_{\star})^{k+1}$ we observe that for $s$ sufficiently close to $s_{\star}$ we can write 
\[
	s^{\alpha} 
	= s_{\star}^{\alpha} \left( \frac{s-s_{\star}}{s_{\star}} + 1 \right)^{\alpha}
	= s_{\star}^{\alpha} + s_{\star}^{\alpha} \sum_{j=1}^{\infty} \binom{\alpha}{j} \frac{(s-s_{\star})^{j}}{s_{\star}^j} ,
\]
and, hence, since $s_{\star}^{\alpha}=z$ we have
\[
	s^{\alpha} - z 
	= s_{\star}^{\alpha} \sum_{j=1}^{\infty} \binom{\alpha}{j} \frac{(s-s_{\star})^{j}}{s_{\star}^j} .
\]

After observing that $\binom{\alpha}{1}=\alpha$, it is a direct computation to provide the expansion
\begin{eqnarray*}
	H_{k}(s;z) (s-s_{\star})^{k+1} 
	&=& s^{\alpha-\beta} \left( \frac{s^{\alpha} - z}{s-s_{\star}} \right)^{-(k+1)} \\
	&=& \frac{s_{\star}^{-(\alpha-1)(k+1)} }{\alpha^{k+1}} s^{\alpha-\beta}  
		\left( 1 + \frac{1}{\alpha} \sum_{j=1}^{\infty} \binom{\alpha}{j+1} \frac{  (s-s_{\star})^{j}}{  s_{\star}^j} \right)^{-(k+1)} .\
\end{eqnarray*}

The reciprocal of the $(k+1)$-th power of the unitary formal power series can be evaluated by applying the Miller's formula \cite[Theorem 1.6c]{Henrici1974} thanks to which it is possible to provide the expansion 
\[
	H_{k}(s;z) (s-s_{\star})^{k+1}
	= \frac{s_{\star}^{-(\alpha-1)(k+1)} }{\alpha^{k+1}} s^{\alpha-\beta} \sum_{j=0}^{\infty} H^{(k)}_{j} \frac{(s-s_{\star})^{j}}{s_{\star}^j} .
\]

To evaluate (\ref{eq:ResidueDifferentiation}) we now introduce, for convenience, the functions
\[
	F_{1}(s) = e^{s}
	, \quad
	F_{2}(s) = s^{\alpha-\beta}
	, \quad
	F_{3}(s) = \sum_{j=0}^{\infty} H^{(k)}_{j} \frac{(s-s_{\star})^{j}}{s_{\star}^j} ,
\]
and hence, thanks to (\ref{eq:ResidueDifferentiation}) we are able to write
\[
	\frac{d^{k}}{ds^{k}}  \left( e^{s}H_{k}(s;z) (s-s_{\star})^{k+1} \right) = 
	\frac{s_{\star}^{-(\alpha-1)(k+1)}}{\alpha^{k+1}}  \sum_{k_1=0}^{k} \sum_{k_2=0}^{k-k_{1}} \frac{k! F_{1}^{(k_1)}(s)  F_{2}^{(k_2)}(s)  F_{3}^{(k-k_1-k_2)}(s) }{k_1!  k_2!  (k-k_1-k_2)!} .
\]

It is immediate to evaluate the derivatives of the functions $F_1$, $F_2$ and $F_3$ and their limit as $s\to s_{\star}$
\[
	\lim_{s \to s_{\star}} F_{1}^{(j)}(s) = e^{s_{\star}}
	, \quad
	\lim_{s \to s_{\star}} F_{2}^{(j)}(s) = (\alpha-\beta)_{j} s_{\star}^{\alpha-\beta-j}
	, \quad
	\lim_{s \to s_{\star}} F_{3}^{(j)}(s) = \frac{j! \, H^{(k)}_{j}}{s_{\star}^{j}} ,
\]
and hence we are able to compute
\[
	\operatorname{Res} \bigl( e^{s}H_{k}(s;z),s_{\star} \bigr) 
	= \frac{s_{\star}^{-(\alpha-1)(k+1)}}{\alpha^{k+1}} 
		\sum_{k_1=0}^{k} \sum_{k_2=0}^{k-k_{1}} 
		\frac{ e^{s_{\star}} (\alpha-\beta)_{k_{2}} s_{\star}^{\alpha-\beta-k_{2}} H^{(k)}_{k-k_1-k_2} }{k_1!  k_2! s_{\star}^{k-k_1-k_2} } 
\]
from which we obtain
\[
	\operatorname{Res} \bigl( e^{s}H_{k}(s;z),s_{\star} \bigr) 
	= \frac{s_{\star}^{1 - \alpha k - \beta }}{\alpha^{k+1}} e^{s_{\star}}
		\sum_{k_1=0}^{k} \frac{s_{\star}^{k_{1}}}{k_1! } \sum_{k_2=0}^{k-k_{1}} \frac{ (\alpha-\beta)_{k_{2}}  H^{(k)}_{k-k_1-k_2} }{ k_2! } 
\]
and the proof immediately follows. \qed
\end{proof}

We observe that, with respect to the formula presented in \cite{TomovskiPoganiSrivastava2014}, the result of Proposition \ref{prop:ResiduesH} is slightly more general and easier to be used for computation. The coefficients of polynomials $P_{k}(x)$ can be evaluated without any particular difficulty by means of a simple algorithm. For ease of presentation we show here their first few instances 
\begin{eqnarray*}
	P_{0}(x) &=& 1 \\
	P_{1}(x) &=& (\alpha-\beta+1) + x \\
	P_{2}(x) &=& \frac{1}{2} + \left(\frac{3\alpha}{2} - \beta + \frac{3}{2}\right) x 
	+ \left( \alpha^2 - \frac{3\alpha\beta}{2} + \frac{3\alpha}{2} + \frac{\beta^2}{2} - \beta + \frac{1}{2} \right) x^2 \
\end{eqnarray*}

We have to consider that the algorithm devised in \cite{Garrappa2015_SIAM} for the inversion of the LT of the ML function needs some adjustment to properly evaluate the derivatives. As $k$ increases a loss of accuracy can indeed be expected for very small arguments also on the negative real semi axis (namely the branch-cut). This is a consequence of the fact that the origin actually behaves as a singularity, of order $k+1$, and hence its effects must be included in the balance of the error. 

Moreover, for high order derivatives, round-off errors must be expected especially when a singularity of $H_k(s;z)$ lies not so close to the origin (thus to make very difficult the choice of a contour encompassing all of them) but at the same time not so far from the origin (thus, making difficult to select contour in the region between the origin and the singularity). To safely deal with situations of this kind we will introduce, later in Subsection \ref{SS:CombinedAlgorithm}, a technique which keeps as low as possible the order of the derivatives to be evaluated by the numerical inversion of the Laplace transform.

\subsection{Summation formulas}\label{SS:Dzherbashyan}

The Djrbashian's formula \cite{Djrbashian1993} allows to express the first derivative of the ML function in terms of two instances of the same function, according to
\begin{equation}\label{eq:FirstDerivative}
		\frac{d}{dz} 	E_{\alpha,\beta}(z) = \frac{E_{\alpha,\beta-1}(z) + (1-\beta)E_{\alpha,\beta}(z)}{\alpha z} ,
\end{equation}
and, obviously, (\ref{eq:FirstDerivative}) holds only for $z\not=0$; anyway, from (\ref{eq:ML_Deriv}) it is immediate to verify that 
\begin{equation}\label{eq:ML_Deriv0}
	\left. \frac{d^k}{dz^k} E_{\alpha,\beta}(z) \right|_{z=0} = \frac{k!}{\Gamma(\alpha k + \beta)} .
\end{equation}


Equation (\ref{eq:FirstDerivative}) provides an exact formulation for the first-order derivative of the ML function but we are able to generalize it also to higher-order derivatives by means of the following result.
\begin{prop}[Summation formula of Djrbashian type]\label{prop:ML_Der_Recursive}
Let $\alpha>0, \beta \in \Rset$ and $z\not=0$. For any $k \in \Nset$ it is
\begin{equation}\label{eq:ML_Der_Recursive}
	\frac{d^{k}}{dz^{k}} E_{\alpha,\beta}(z) = \frac{1}{\alpha^k z^{k}} \sum_{j=0}^{k} c_{j}^{(k)} E_{\alpha,\beta-j}(z) ,
\end{equation}
where $c_{0}^{(0)}=1$ and the remaining coefficients $c_{j}^{(k)}$, $j=0,1,\dots,k$, are recursively evaluated as 
\begin{equation}\label{eq:Coeff_Dzherbashyan}
	c_{j}^{(k)} = 
	\left\{ \begin{array}{ll}
		(1 - \beta - \alpha (k-1)) c_{0}^{(k-1)} & j = 0 \\
		c_{j-1}^{(k-1)} + (1 - \beta - \alpha (k -1) + j) c_{j}^{(k-1)} \quad & j=1,\dots,k-1 \\
		1 & j = k \
	\end{array}\right.
\end{equation}	
\end{prop}

\begin{proof}
We proceed by induction on $k$. For $k=0$ the proof is obvious and for $k=1$ it is a consequence of (\ref{eq:FirstDerivative}). Assume now that (\ref{eq:ML_Der_Recursive}) holds for $k-1$; then, the application of standard derivative rules, together with the application of (\ref{eq:FirstDerivative}), allows us to write
\begin{eqnarray*}
	\frac{d^{k}}{dz^{k}} E_{\alpha,\beta}(z) 
	&=&
	\frac{1}{\alpha^{k-1} z^{k-1}} \sum_{j=0}^{k-1} \frac{c_{j}^{(k-1)}}{\alpha z} \Bigl( E_{\alpha,\beta-j-1}(z) - (\beta-j-1)E_{\alpha,\beta-j}(z) \Bigr) \\
	& & - \frac{(k-1) z^{k-2}}{\alpha^{k-1} z^{2k-2}} \sum_{j=0}^{k-1} c_{j}^{(k-1)} E_{\alpha,\beta-j}(z)  \
\end{eqnarray*}
and the proof follows after defining the coefficients $c_{j}^{(k)}$ as in (\ref{eq:Coeff_Dzherbashyan}).  \qed
\end{proof}

The Djrbashian summation formula (SF) of Proposition \ref{prop:ML_Der_Recursive} allows to represent the derivatives of the ML function in terms of a linear combination of values of the same function; clearly it can be employed in actual computation once a reliable procedure for the evaluation of $E_{\alpha,\beta}(z)$ for any parameter $\alpha$ and $\beta$ is available, as it is the case of the method described in \cite{Garrappa2015_SIAM}.

Unfortunately, it seems not possible to provide an explicit closed form for the coefficients $c_{j}^{(k)}$ in the Djrbashian SF (\ref{eq:ML_Der_Recursive}); anyway the recursive relationship (\ref{eq:Coeff_Dzherbashyan}) is very simple to compute and the first few coefficients, holding up to the derivative of third order, are shown in Table \ref{tbl:Coeff_Dzherbashyan}.

\begin{table}
\footnotesize
\[
	\begin{array}{l|c|c|c|c|} \hline
			& k=0 & k=1 & k=2 & k=3 \\ \hline
		c_{0}^{(k)} & 1 & 1-\beta & (1-\beta)(1-\beta-\alpha) & (1-\beta)(1-\beta-\alpha)(1-\beta-2\alpha) \\
		c_{1}^{(k)} &   & 1       & 3-2\beta-\alpha & (1-\beta)(1-\beta-\alpha) + (3-2\beta-\alpha) (2-\beta-2\alpha) \\
		c_{2}^{(k)} &   &         &               1 & 6-3\beta-3\alpha \\
		c_{3}^{(k)} &   &         &                 & 1 \\ \hline	
	\end{array}
\]
\normalsize
\caption{First few coefficients (\ref{eq:Coeff_Dzherbashyan}) of the summation formulas (\ref{eq:ML_Der_Recursive}) and (\ref{eq:ML_Der_Recursive2}).}\label{tbl:Coeff_Dzherbashyan}
\end{table}

An alternative approach to compute the coefficients $c_{j}^{(k)}$ can be derived by observing, from (\ref{eq:ML_Deriv0}), that the derivatives of the ML function have finite values at $z=0$; hence, since $E_{\alpha,\beta}$ is entire, all possible negative powers of $z$ in (\ref{eq:ML_Der_Recursive}) must necessarily vanish and, since $c_{k}^{(k)}=1$, it is immediate to verify that the coefficients $c_{j}^{(k)}$ are solution of the linear system

\[
	\sum_{j=0}^{k-1} \frac{1}{\Gamma(\alpha \ell + \beta-j)}c_{j}^{(k)} = -\frac{1}{\Gamma(\alpha \ell + \beta-k)}
	, \quad
	\ell = 0, \dots k-1 .
\]

To study the stability of the Djrbashian SF (\ref{eq:ML_Der_Recursive}), we observe that since the derivatives of the ML function have finite values at $z=0$, see Eq. (\ref{eq:ML_Deriv0}), from (\ref{eq:ML_Der_Recursive}) it follows 
\begin{equation}\label{eq:DjrbashianCloseZero}
	\sum_{j=0}^{k} c_{j}^{(k)} E_{\alpha,\beta-j}(z) = {\cal O} \bigl(z^{k}\bigr) 
	, \quad |z|\to 0 .
\end{equation}

As a consequence, differences of almost equal values are involved by (\ref{eq:ML_Der_Recursive}) when $|z|$ is small; the unavoidable numerical cancellation, further magnified since the division by $z^{k}$, makes this summation formula not reliable for values of $z$ close to the origin, especially for derivatives of high order.

To overcome these difficulties we can exploit again the relationship (\ref{eq:MLder_ML3}) between the ML derivatives and the Prabhakar function together with the formula (see \cite{Prabhakar1971})
\begin{equation}\label{eq:PrabhakarReductionGamma}
	E_{\alpha,\beta}^{k+1}(z) = 
	\frac{E_{\alpha,\beta-1}^{k}(z) + (1-\beta+\alpha k) E_{\alpha,\beta}^{k}(z)}{\alpha k} ,
\end{equation}
expressing the Prabhakar function as difference of two instances of the same function but with a smaller value of the third parameter. We are hence able to provide the following alternative SF of Prabhakar type.

\begin{prop}[Summation formula of Prabhakar type]\label{prop:ML_Der_Recursive2}
Let $\alpha>0$ and $\beta \in \Rset$. For any $k \in \Nset$ it is
\begin{equation}\label{eq:ML_Der_Recursive2}
	\frac{d^{k}}{dz^{k}} E_{\alpha,\beta}(z) = \frac{1}{\alpha^k } \sum_{j=0}^{k} c_{j}^{(k)} E_{\alpha,\alpha k + \beta-j}(z) ,
\end{equation}
where $c_{j}^{(k)}$, $j=0,1,\dots,k$, are the coefficients (\ref{eq:Coeff_Dzherbashyan}) of Proposition \ref{prop:ML_Der_Recursive}.
\end{prop}

\begin{proof}
The claim is obvious for $k=0$; for $k=1$ it is a consequence of (\ref{eq:PrabhakarReductionGamma}) and then one can easily prove it by induction with algebraic transformations quite similar to the ones used to prove (\ref{eq:ML_Der_Recursive}).  \qed
\end{proof}

For small values of $|z|$, the Prabhakar SF (\ref{eq:ML_Der_Recursive2}) is expected to be less affected by round-off errors than (\ref{eq:ML_Der_Recursive}). Indeed not only the sum in (\ref{eq:ML_Der_Recursive2}) returns ${\cal O}(1)$ when $|z|\to 0$ but possible round-off errors are not amplified by division by $z^k$ as in (\ref{eq:ML_Der_Recursive}). 

In Figures \ref{fig:Fig_SF_Comparison_al06_be=10_th10} and \ref{fig:Fig_SF_Comparison_al08_be=12_th05} we compare the errors given by the SF (\ref{eq:ML_Der_Recursive}) of Djrbashian type with those provided by the SF (\ref{eq:ML_Der_Recursive2}) of Prabhakar type; here and in the following tests, the reference values for the derivatives of $E_{\alpha,\beta}(z)$ are obtained by computing the series expansion (\ref{eq:ML_Deriv}) in high-precision floating-point arithmetic with 2000 digits by using Maple 15. As we can clearly see, when $|z|$ is small the Prabhakar SF (\ref{eq:ML_Der_Recursive2}) is surely more accurate whilst for larger values of $|z|$ their performances are almost similar. 

\begin{figure}[ht]
\centering
\resizebox{0.60\hsize}{!}{\includegraphics{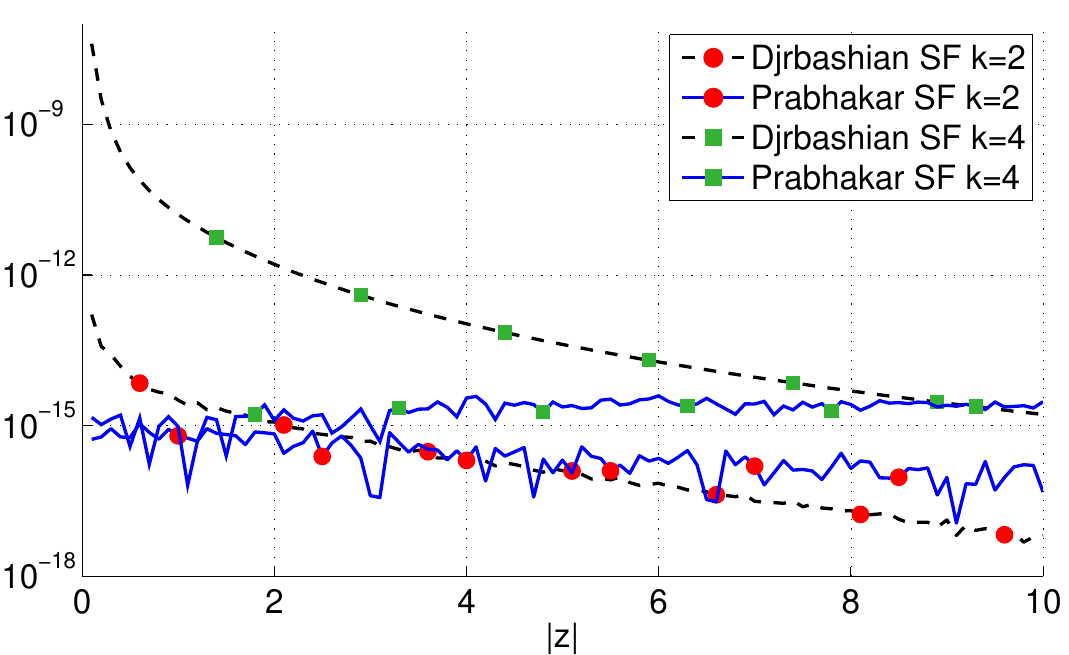} }
\caption{Comparison of errors provided by Djrbashian and Prabhakar SFs for $\alpha=0.6$, $\beta=1.0$ and $\arg(z)=\pi$.}
\label{fig:Fig_SF_Comparison_al06_be=10_th10}
\end{figure}

\begin{figure}[ht]
\centering
\resizebox{0.60\hsize}{!}{\includegraphics{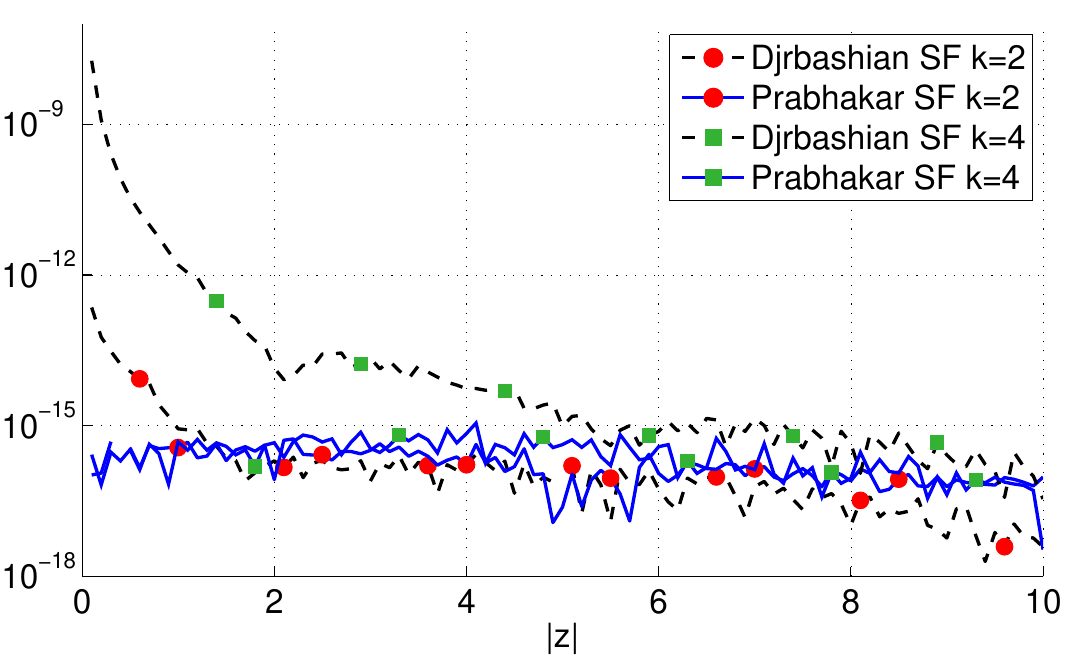}}
\caption{Comparison of errors provided by Djrbashian and Prabhakar SFs for $\alpha=0.8$, $\beta=1.2$ and $\arg(z)=\pi/2$.}
\label{fig:Fig_SF_Comparison_al08_be=12_th05}
\end{figure}

\subsection{Combining algorithms with derivatives balancing}\label{SS:CombinedAlgorithm}

The mathematical considerations in the previous subsection, supported by several experiments we conducted on a wide range of parameters $\alpha$ and $\beta$, and for arguments $z$ in different sectors of the complex plane, indicate that the Prabhakar SF (\ref{eq:ML_Der_Recursive2}) is capable of providing accurate and satisfactory results in the majority of the cases. 

Anyway we have to consider that larger errors should be expected as the degree of the derivative increases; this a consequence of the large number of operations involved by (\ref{eq:ML_Der_Recursive2}) and, in particular, of the accumulation of round-off errors. Unfortunately, also the numerical inversion of the LT can show a loss of accuracy for high order derivatives since the strong singularity in $H_k(s;z)$ may lead to large round-off errors in the integration process.

In critical situations it is advisable to adopt mixed strategies by combining different algorithms. For instance, it is natural to use the truncated series expansion (\ref{eq:ML_DerivTruncated}) for small arguments and switch to other methods as $|z|$ becomes larger; the error estimates discussed in Subsection \ref{sec:series} can be of help for implementing an automatic procedure performing this switching. 

In addition, since all methods reduce their accuracy for high order derivatives, it could be necessary to adopt a strategy operating a balance of the derivatives thus to bring the problem to the evaluation of lower order derivatives. To this purpose, Proposition \ref{prop:ML_Der_Recursive2} can be generalized in order to express $k$-th order derivatives of the ML function in terms of lower order derivatives instead of the ML function itself. 

\begin{prop}\label{prop:ML_Der_mixed}
Let $\alpha>0, \beta \in \Rset$. For any $k \in \Nset$ and $p\leq k$ it is
\begin{equation}\label{eq:ML_Der_mixed}
	\frac{d^{k}}{dz^{k}} E_{\alpha,\beta}(z) = \frac{1}{\alpha^{k-p} } \sum_{j=0}^{k-p} c_{j}^{(k-p)} \frac{d^{p}}{dz^{p}} E_{\alpha,(k-p)\alpha  + \beta-j}(z) ,
\end{equation}
where $c_{j}^{(k)}$, $j=0,1,\dots,k$, are the coefficients (\ref{eq:Coeff_Dzherbashyan}) of Proposition \ref{prop:ML_Der_Recursive}.
\end{prop}

The proof of the above result is similar to the proof of Proposition \ref{prop:ML_Der_Recursive2} and is omitted for shortness.

Thanks to Proposition \ref{prop:ML_Der_mixed} it is possible to combine the algorithms introduced in the previous subsections and employ the SF (\ref{eq:ML_Der_mixed}) with underlying derivatives evaluated by the numerical inversion of the LT. Clearly, the aim is to avoid working with high-order derivatives in the numerical inversion of the LT and, at the same time, reduce the number of terms in (\ref{eq:ML_Der_Recursive2}).

In Figure \ref{fig:Fig_ML_BalancDeriv_al06_be10_theta05_k5} we illustrate the positive effects on accuracy of performing the derivatives balancing. As we can see, the SF (\ref{eq:ML_Der_Recursive2}) tends to lose accuracy for higher order derivatives (here a $5$-th order has been considered); the numerical inversion of the LT is instead more accurate but when the integration contour is constrained to stay close to the singularity the error can rapidly increase (see the peak in the interval $[0,2]$). The derivatives balancing (operated here for $p=1$) allows to keep the error at a moderate level on the whole interval, without any worrisome peak.

\begin{figure}[ht]
\centering
\resizebox{0.6\hsize}{!}{\includegraphics{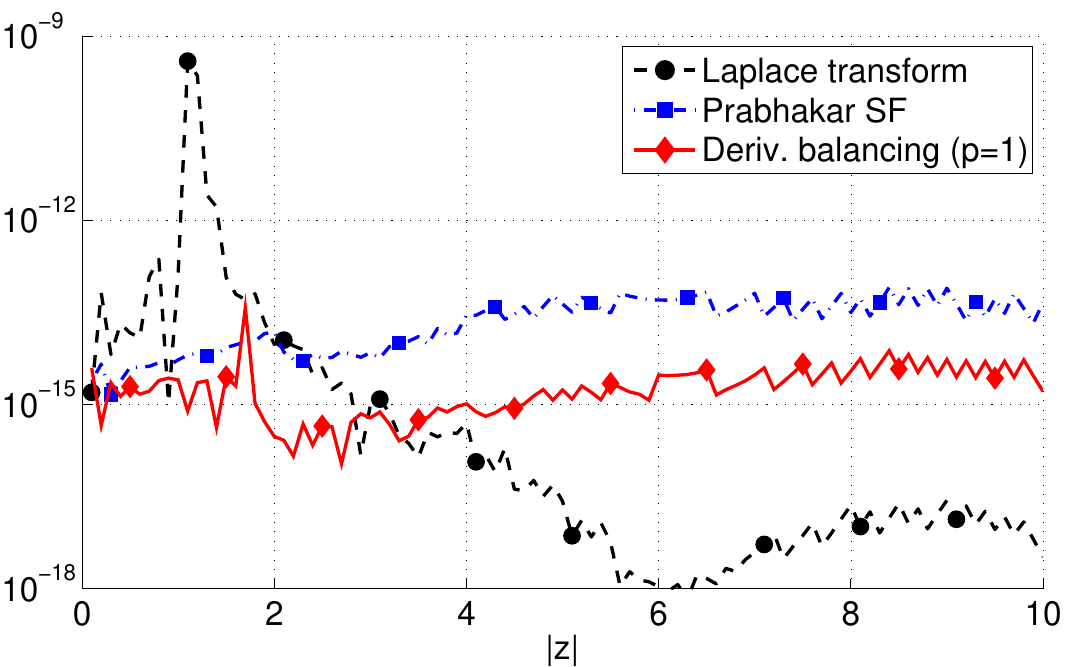}}
\caption{Comparison of errors for $5$-th derivative of the ML function when $\alpha=0.6$, $\beta=1.0$ and $\arg(z)=\pi/2$.}
\label{fig:Fig_ML_BalancDeriv_al06_be10_theta05_k5}
\end{figure}

Finally, we present a couple of tests in which the different techniques described in the paper are combined to provide accurate results for the evaluation of the derivatives of the ML function. Roughly speaking, the algorithm uses the truncated series for small arguments and hence, on the basis of the error estimate, it switches to the Prabhakar SF (\ref{eq:ML_Der_Recursive2}) using the numerical inversion of the Laplace transform according to a derivative balancing approach.

In several cases we have observed that the computation of lower order derivatives turns out to be more accurate (see Figure \ref{fig:Fig_ML_Der_Test_al06_be=06_th08}) although it is not possible to infer a general trend as we can observe from Figure \ref{fig:Fig_ML_Der_Test_al08_be=12_th05}. Anyway, in all cases (we have performed experiments, not reported here for brevity, on a wide set of arguments and parameters) the errors remain in a range $10^{-13} \sim 10^{-15}$ which can be surely considered satisfactory for applications. The plotted error is $|E-\tilde{E}|/(1+|E|)$ where $E$ is a reference value and $\tilde{E}$ the evaluated derivative. 

\begin{figure}[ht]
\centering
\resizebox{0.60\hsize}{!}{\includegraphics{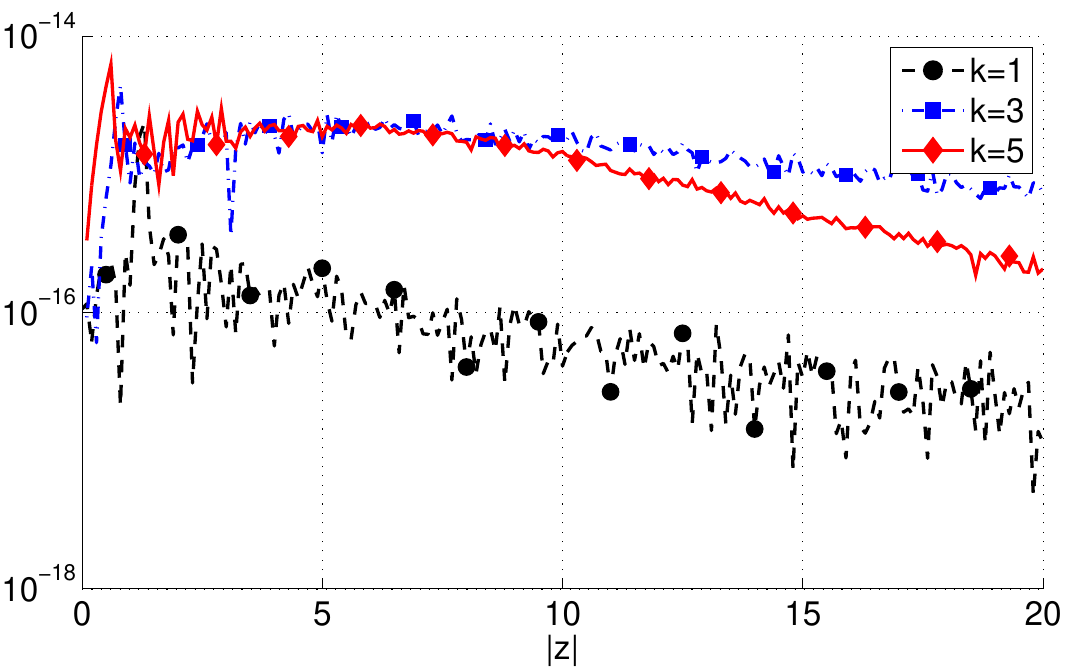}}
\caption{Errors in the computation of $k$-th derivatives of the ML function when $\alpha=0.6$, $\beta=0.6$ and $\arg(z)=0.8\pi$.}
\label{fig:Fig_ML_Der_Test_al06_be=06_th08}
\end{figure}

\begin{figure}[ht]
\centering
\resizebox{0.60\hsize}{!}{\includegraphics{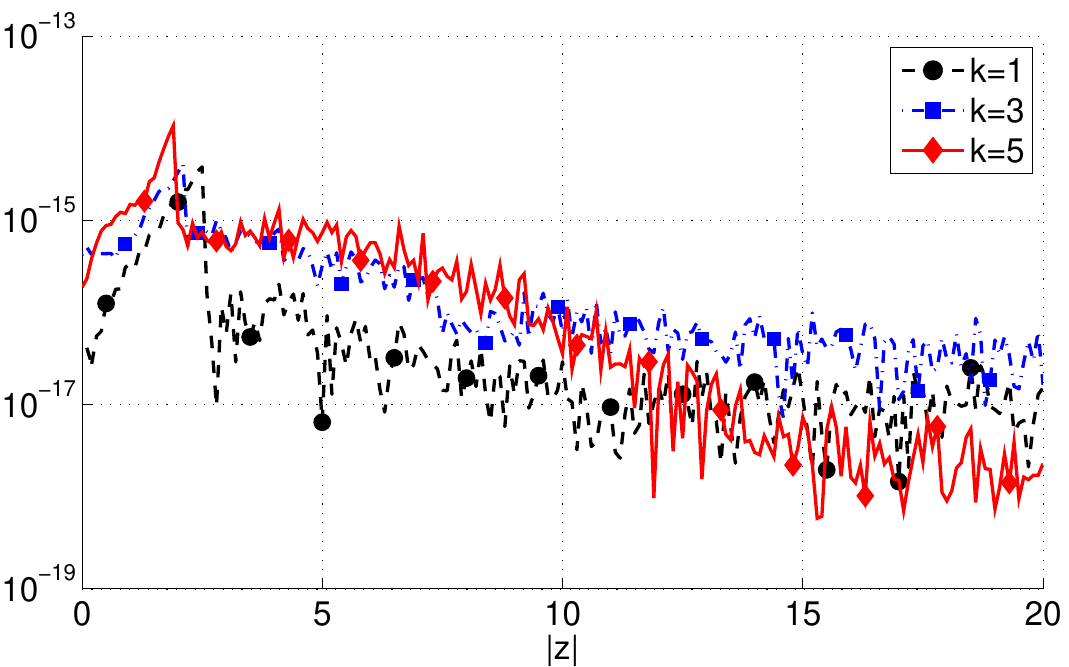}}
\caption{Errors in the computation of $k$-th derivative of the ML function when $\alpha=0.8$, $\beta=1.2$ and $\arg(z)=\pi/2$.}
\label{fig:Fig_ML_Der_Test_al08_be=12_th05}
\end{figure}


\section{Conditioning of the matrix ML function}\label{S:Conditioning}

To hone the analysis of the matrix ML function it is of fundamental importance to measure its sensitivity to small perturbations in the matrix argument; indeed, even when dealing with exact data, the rounding errors can severely affect the computation, exactly as in the scalar case. A useful tool for this analysis is the {\emph{condition number}} whose definition can be readily extended to matrix functions \cite{Higham2008}.

\begin{defn} 
For a general matrix function $f$, if $\|\cdot\|$ denotes any matrix norm, the {\emph{absolute condition number}} is defined as 
\[
\condabs(f,A)= \lim_{\varepsilon \rightarrow 0} \sup_{\|E\|\leq \varepsilon} \frac{\|f(A+E)-f(A)\|}{\varepsilon}
\]
while the {\emph{relative condition number}} is defined as 
\[
\condrel(f,A) = \lim_{\varepsilon \rightarrow 0} \sup_{\|E\|\leq \varepsilon\|A\|} \frac{\|f(A+E)-f(A)\|}{\varepsilon\|f(A)\|}
\]
being the matrix $E$ a perturbation to the original data $A$. 
\end{defn}

For the actual computation of these quantities the Fr\'echet derivative of the matrix function $f$ is of use. We recall that given a matrix function $f$ defined in $\Cset^{n\times n}$, its  Fr\'echet derivative at a point $A\in \Cset^{n\times n}$  is a linear function $L(A,E)$ such that 
\[
	f(A+E)-f(A) -L(A,E) = o(\|E\|).
\]

The norm of the Fr\'echet derivative is defined as
\[
\|L(A)\| := \max_{Z\neq 0} \frac{\|L(A,Z)\| }{\|Z\|}
\]
and it turns out to be a useful tool for the sensitivity analysis of the function itself thanks to the following result.

\begin{thm}  \cite{Higham2008}
The absolute and relative condition numbers are given by
\[
\condabs(f,A)=\|L(A)\|,
\]
\[
\condrel(f,A)=\frac{\|L(A)\| }{\|f(A)\|} \|A\|.
\]
\end{thm}

The computation of the Fr\'echet derivative is thus essential to derive the function conditioning. Nowadays only specific functions have effective algorithms to compute it, like the exponential, the logarithm and fractional powers \cite{DieciPapini2000,Higham2008}. For general matrix functions the reference algorithm for the Fr\'echet derivative  is due to Al-Mohy and Higham \cite{AlMohyHigham2010,HighamAlMohy2010}.
 
To give a flavor of the conditioning of the ML function we consider $25$ test matrices of dimension $10\times 10$ taken from the Matlab {\tt{gallery}}. Figure \ref{fig:Fig_ML_Matrix_Conditioning_Abs} reports the $1$-norm relative condition number $\kappa_r(E_{\alpha,1},A)$
computed by the {{\tt{funm{\_}condest1}}} routine from the Matrix Function Toolbox \cite{Higham2008}.  Amazingly, for all but the {chebspec} matrix (the one plotted in red line with small circles), the conditioning is almost independent on $\alpha$; this shows that the ML function has almost the same conditioning of the exponential function for any $\alpha$. For most of the test matrices the ML function conditioning is large; in particular, for small values of $\alpha$ the ML conditioning for the chebspec Matlab matrix is huge. An explanation could be that the Schur-Parlett algorithm does not work properly with this matrix, as stressed in \cite{Higham2008}; indeed, this matrix is similar to a Jordan block with eigenvalue 0 while its computed eigenvalues lie roughly on a circle with centre 0 and radius 0.2. The well-known ill-conditioning of Pascal matrices (for $n=10$ it is $\kappa(A)\approx 8.13\times10^9$) clearly leads to a severe ill-conditioning of the matrix function too (see the straight blue line).

In general, more insights can derive from a comparison of this conditioning with the $1$-norm condition of $A$, that we denote with $\kappa(A)$. To this purpose, Figure \ref{fig:Fig_ML_Matrix_Conditioning_Rel} displays the ratio $\kappa_r(E_{\alpha,1},A)/\kappa(A)$. From this plot we can appreciate that in most of the cases the ML function is better conditioned than the matrix argument itself, while in the remaining cases the conditioning is not too much amplified.

\begin{figure}[ht]
\centering
\resizebox{0.60\hsize}{!}{\includegraphics{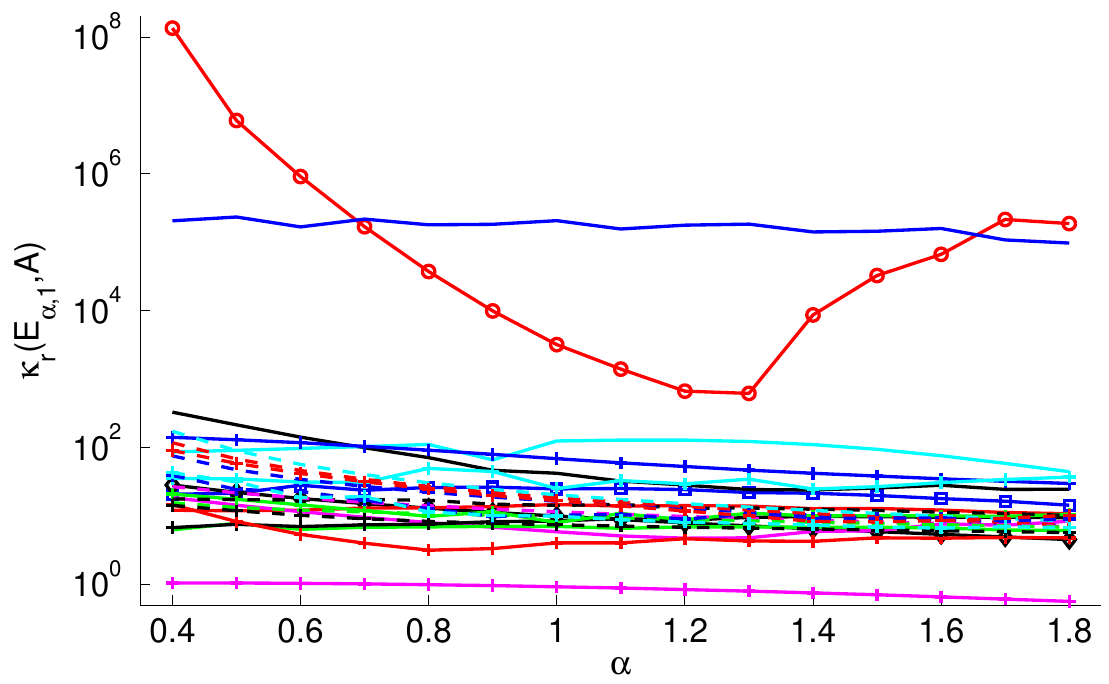}}
	\caption{Conditioning $\kappa_r(E_{\alpha,1},A)$ for $25$ matrices of dimension $10\times 10$ from the Matlab gallery. } 
	\label{fig:Fig_ML_Matrix_Conditioning_Abs} 
\end{figure}

\begin{figure}[ht]
\centering
\resizebox{0.60\hsize}{!}{\includegraphics{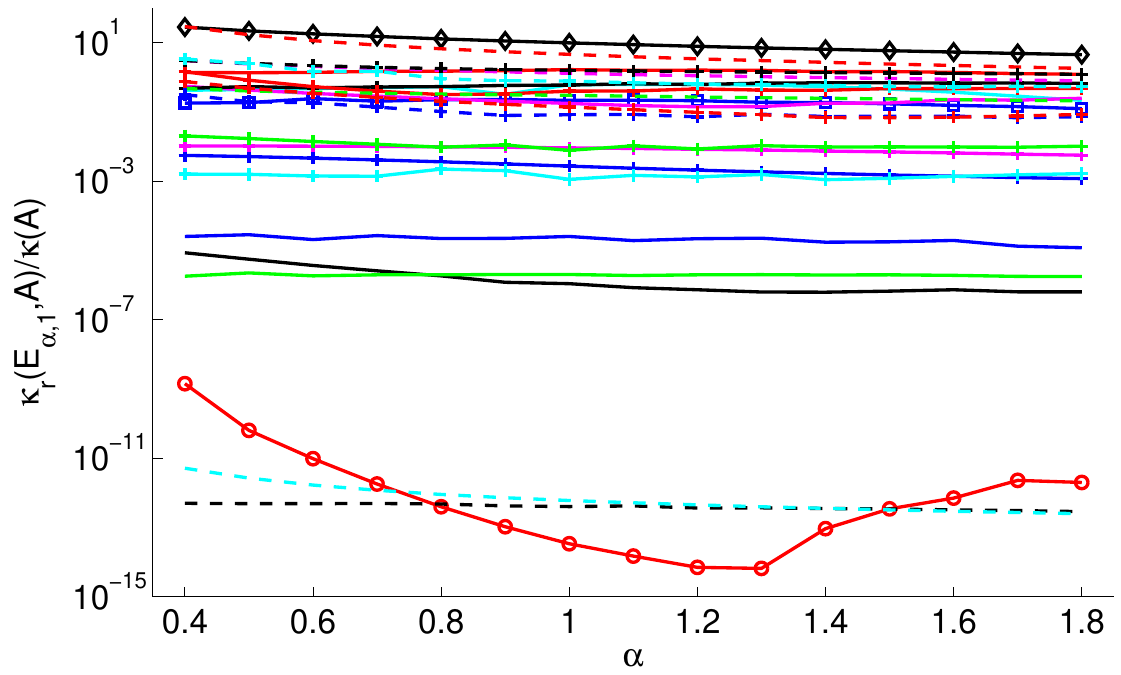}}
	\caption{Conditioning ratio $\kappa_r(E_{\alpha,1},A)/\kappa(A)$ for $25$ matrices of dimension $10\times 10$ from the Matlab gallery. } 
	\label{fig:Fig_ML_Matrix_Conditioning_Rel} 
\end{figure}

Moreover in Figure \ref{fig:Fig_ML_SelectedEigenvalues} of the next section we show that $\condrel(E_{\alpha,1},A)\epsilon$ well bounds the error between the computed and the reference values of $E_{\alpha,1}(A)$; this property seems to reveal  that our technique to compute $E_{\alpha,1}(A)$ is {\emph{normwise forward stable}}, according to the definition given in \cite{Higham2002}.


\section{Numerical experiments}\label{S:Experiments}

In addition to the numerical tests already presented in the previous sections, in this section we verify the accuracy obtained for some test matrices. All the experiments have been carried out in Matlab ver. 8.3.0.532 (R2014a) and, in order to plot the errors, also in these experiments reference values have been evaluated by means of Maple 15 with a floating-point precision of 2000 digits.

For the Schur-Parlett algorithm we use the Matlab {\tt funm} function, whilst the derivatives of the ML function are evaluated by means of the combined algorithm described in Subsection \ref{SS:CombinedAlgorithm} and making use of the derivatives balancing.  The corresponding Matlab code, which will be used in all the subsequent experiments, is available in the file exchange service of the Mathworks website\footnotemark\footnotetext{www.mathworks.com/matlabcentral/fileexchange/66272-mittag-leffler-function-with-matrix-arguments}. The plotted errors between reference values $E$ and approximations $\tilde{E}$ are evaluated as $\|E-\tilde{E}\|/(1+\|E\|)$, where $\| \cdot \|$ is the usual Frobenius norm.

\subsection*{Example 1: the Redheffer matrix}

For the first test we use the Redheffer matrix from the matrix gallery of Matlab. This is a quite simple binary matrix whose elements $a_{i,j}$ are equal to 1 when $j=1$ or $i$ divides $j$, otherwise they are all equal to 0; however, the eigenvalues are highly clustered: for a $n\times n$ Redheffer matrix the number of eigenvalues equal to 1 is exactly $n- \left\lfloor \log_2 n  \right\rfloor -1$ (see \cite{BarrettJarvis1992}). As a consequence, computing $E_{\alpha,\beta}(-A)$ demands the evaluation of high order derivatives (up to the 24-th order in our experiments) thus making this matrix appealing for test purposes.

In Figure \ref{fig:Fig_RedHeff_Matrix} we observe the error for some values of $\alpha$ with $n \times n$ Redheffer matrices of increasing size ($n$ is selected in the interval $[4,20]$ and $\beta=1$ is always used). As we can clearly see, the algorithm is able to provide an excellent approximation, with an error very close to machine precision, also for matrices of notable size.   

\begin{figure}[ht]
\centering
\resizebox{0.60\hsize}{!}{\includegraphics{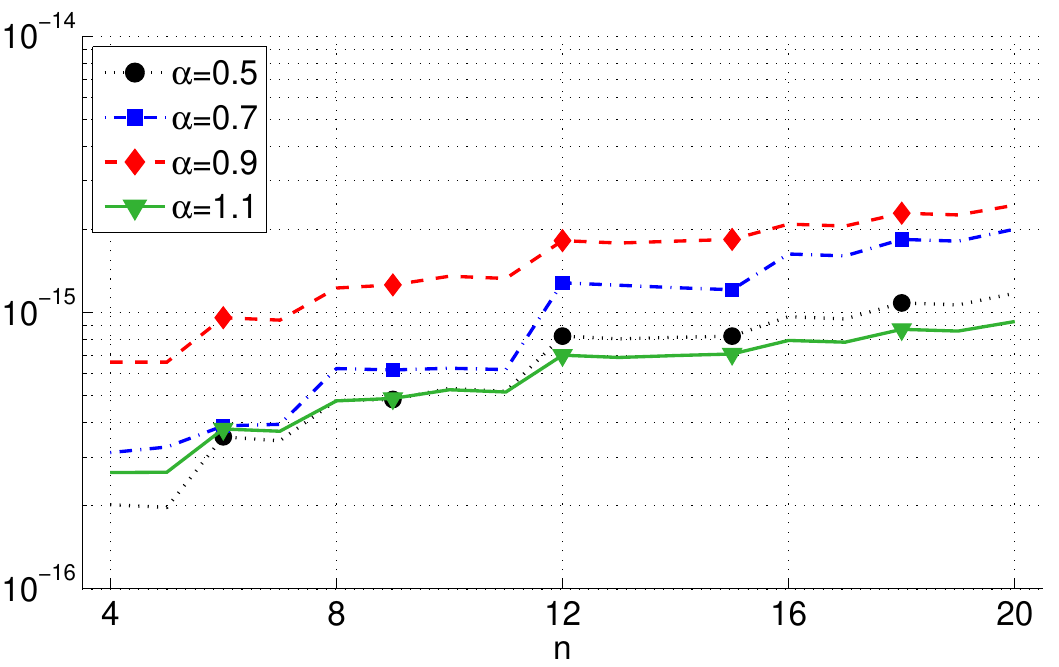}}
\caption{Errors in the computation of $E_{\alpha,1}(-A)$ with $A$ the Redheffer matrix as its dimension changes}
\label{fig:Fig_RedHeff_Matrix}
\end{figure}


\subsection*{Example 2: matrices with selected eigenvalues}\label{SS:selectedEigs}

The second experiment involves 4 different matrices of size  $40 \times 40$ suitably built in order to have eigenvalues with moderately high multiplicities. In practice we fix some values and we consider diagonal matrices having them as principal entries, repeated according to the multiplicities we want (as listed in Table \ref{tbl:EigMulMatrices40}). Then, by  similarity transformations, we get the full matrices with the desired spectrum. To reduce rounding errors for the similarity transformations  we use orthogonal matrices provided by the {{\tt{gallery}} Matlab function.

\begin{table}[h!]
\footnotesize
\centering
\begin{tabular}{l|c|c} \hline 
 & Eigenvalues (multiplicities) & Max derivative \\
 & & order \\ \hline
Matrix 1 & $\pm{1.0}(5)$ \; $\pm{1.0001}(4)$ \; $\pm 1.001 (4)$ \; $\pm 1.01 (4)$ \; $\pm 1.1 (3)$ & 15 \\ \hline
Matrix 2 & $\pm 1.0 (8)$ \; $2 (8)$ \, $-5 (8)$ \, $-10 (8)$ & 3 \\ \hline
Matrix 3 & $-1(2)$ \; $-5(2)$ \, $1\pm10\iu (6)$ \; $-4\pm1.5\iu (6)$ \; $\pm5\iu (6)$ & 3 \\ \hline
Matrix 4 & $1 (4)$ \; $1.0001 (4)$ \; $1.001 (4)$ \; $1\pm10\iu (7)$ \; $-4\pm1.5\iu (7)$ & 7 \\ \hline
\end{tabular}
\caption{Eigenvalues (with multiplicities) for the $40 \times 40$ matrices of the Example 2. }\label{tbl:EigMulMatrices40}
\normalsize
\end{table}

The maximum order of the derivatives required by the computation is indicated in the last column of Table \ref{tbl:EigMulMatrices40}. As we can observe from Figure \ref{fig:Fig_ML_SelectedEigenvalues} (solid lines), despite the large size of the matrices and the high clustering of the eigenvalues, it is possible to evaluate $E_{\alpha,\beta}(A)$ with a reasonable accuracy. In the same plot we can also appreciate (dotted lines) how the bounds $\condrel(E_{\alpha,1},A)\epsilon$ (see Section \ref{S:Conditioning}) give a reasonable estimate for the error.

\begin{figure}[ht]
\centering
\resizebox{0.60\hsize}{!}{\includegraphics{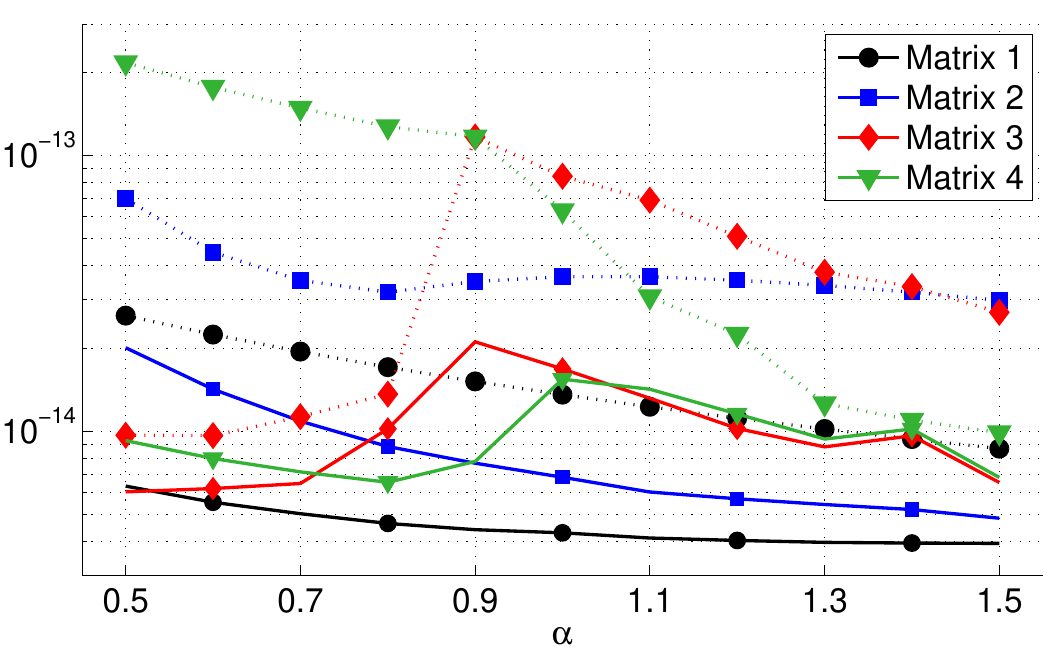}}
\caption{Relative errors (solid lines) for $E_{\alpha,1}(A)$ for the four test matrices $A\in\Rset^{40\times 40}$ of Example 2 and conditioning $\condrel(E_{\alpha,1},A)\epsilon$ (dotted lines).}
\label{fig:Fig_ML_SelectedEigenvalues}
\end{figure}

\subsection*{Example 3: solution of a multiterm FDE}

In the third experiment we consider an application to the multiterm FDE  
\begin{equation}\label{eq:MultiTermFDE_example}
	2 y(t) + 6 D^{\alpha}_{0} y(t) + 7 D^{2\alpha}_{0} y(t) + 4 D^{3\alpha}_{0} y(t) + D^{4\alpha}_{0} y(t) = f(t) 
\end{equation}
with homogeneous initial conditions, which is first reformulated in terms of the linear system (\ref{eq:MultiTermFDESystem}) with a companion coefficient matrix $A$, of size $16\times16$, with both simple and double eigenvalues. The reference solution, for $\alpha=0.8$ and an external source term $f(t) = - t^2/2 + 2t$, is shown in Figure \ref{fig:ML_RefMat_FDEMulti_al08_Sol} on the interval $[0,6]$. 

\begin{figure}[ht]
\centering
\resizebox{0.60\hsize}{!}{\includegraphics{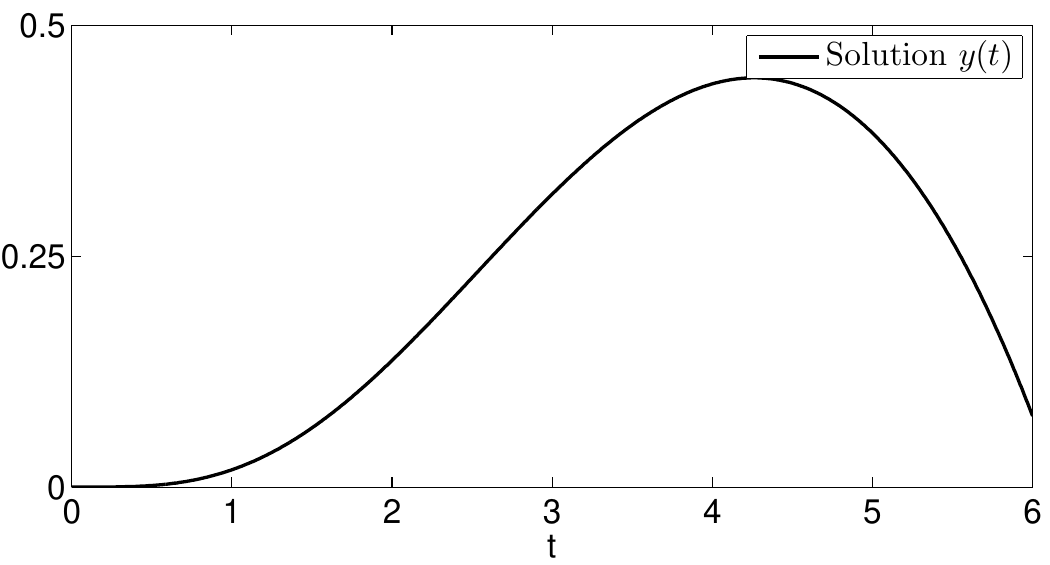}}
\caption{Solution of the multiterm FDE (\ref{eq:MultiTermFDE_example}) for $\alpha=0.8$.}
\label{fig:ML_RefMat_FDEMulti_al08_Sol}
\end{figure}

In Figure \ref{fig:ML_RefMat_FDEMulti_al08_Err} we compare the error of the solution evaluated by using the matrix ML function in (\ref{eq:MultitermFDESpectralSolution}) with those obtained by a step-by-step trapezoidal product integration (PI) rule generalized to multiterm FDEs according the
ideas discussed in \cite{DiethelmLuchko2004}; in the plot we simply denote this approach as ``Trapez. PI'' with indicated the step-size $h$ used for the computation.

\begin{figure}[ht]
\centering
\resizebox{0.60\hsize}{!}{\includegraphics{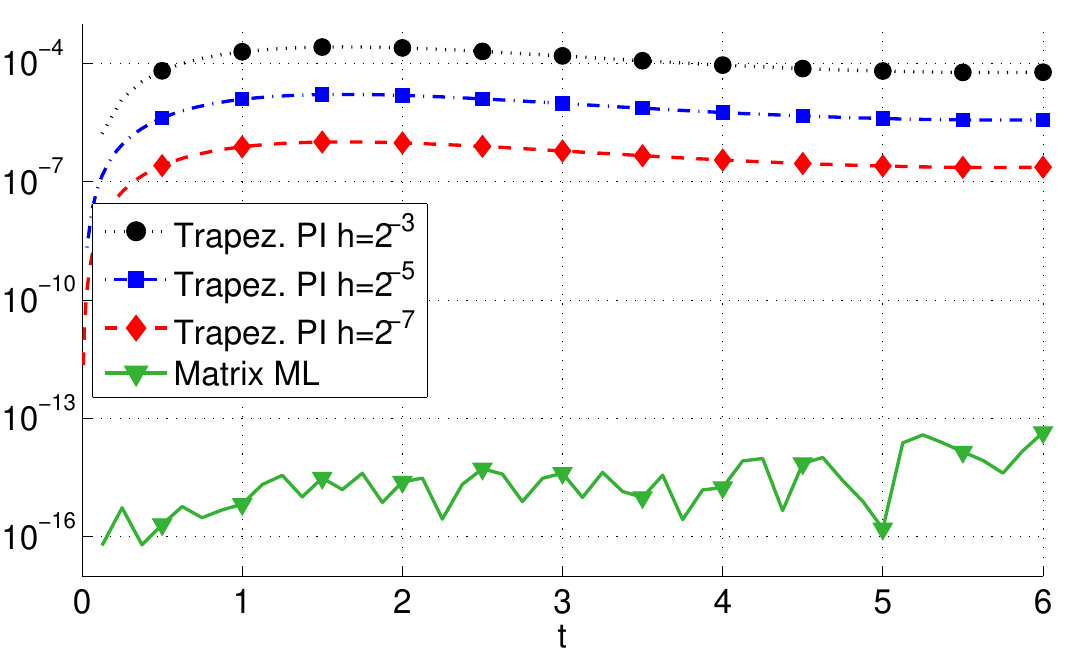}}
\caption{Errors in the solution of the multiterm FDE (\ref{eq:MultiTermFDE_example}) for $\alpha=0.8$.}
\label{fig:ML_RefMat_FDEMulti_al08_Err}
\end{figure}

Also in this case we observe the excellent results obtained thanks to the use of  matrix ML functions. We must also note that in order to achieve a comparable accuracy, the Trapez. PI would require very small step-sizes, with a remarkable computational cost especially for integration on intervals of large size; the use of matrix ML functions instead allows to directly evaluate the solution at any time $t$.

\section{Concluding remarks}

In this paper we have discussed the evaluation of the ML function with matrix arguments and illustrated some remarkable applications. Since one of the most efficient algorithms for the evaluation of matrix functions requires the evaluation of the derivatives of the original scalar function, a large portion of the paper has been devoted to present different methods for the accurate evaluation of derivatives of the ML function, a subject which, as far as we know, has not been faced before except for first derivatives. We have also discussed some techniques for combining the different methods in an efficient way with the aim of devising an algorithm capable of achieving high accuracy with matrices having any kind of spectrum. The analysis on the conditioning has shown that it is possible to keep errors under control when evaluating matrix ML functions. Finally, the numerical experiments presented at the end of the paper have highlighted the possibility of evaluating the matrix ML function with great accuracy also in the presence of not simple matrix arguments.




\end{document}